\documentclass{amsart}
\usepackage{amsthm}
\usepackage{amsmath, ifthen, amssymb}
\usepackage{latexsym}
\usepackage{amsfonts}
\usepackage{graphics}

\begin{document}

\newcommand{\naturalsSymbol}[1]
{
  \ifthenelse{\equal{#1}{N}}
 {\newcommand{\nat}{\ensuremath{\mathbb{N}}}
 }{}
 \ifthenelse{\equal{#1}{omega}}
 {\newcommand{\nat}{\ensuremath{\omega}}
 }{}
}

\naturalsSymbol{omega}

\newcommand{\vg}[1]{}


\newcommand{\n}{\ensuremath{n \in \nat}}
\newcommand{\iin}{\ensuremath{i \in \nat}}
\newcommand{\s}{\ensuremath{s \in \nat}}

\newcommand{\ep}{\ensuremath{\varepsilon}}
\newcommand{\eq}{\ensuremath{\Longleftrightarrow}}

\newcommand{\del}{\ensuremath{\Delta^1_1}}
\newcommand{\sig}{\ensuremath{\Sigma^1_1}}
\newcommand{\pii}{\ensuremath{\Pi^1_1}}
\newcommand{\G}{\ensuremath{\Gamma}}

\newcommand\tboldsymbol[1]{%
\protect\raisebox{0pt}[0pt][0pt]{%
$\underset{\widetilde{}}{\boldsymbol{#1}}$}\mbox{\hskip 1pt}}

\newcommand{\bolddel}{\ensuremath{\tboldsymbol{\Delta}^1_1}}
\newcommand{\boldpii}{\ensuremath{\tboldsymbol{\Pi}^1_1}}
\newcommand{\boldsig}{\ensuremath{ \tboldsymbol{\Sigma}^1_1}}
\newcommand{\boldG}{\ensuremath{\tboldsymbol{\Gamma}}}


\newcommand{\om}{\ensuremath{\omega}}

\newcommand{\C}{\ensuremath{\mathbb C}}
\newcommand{\Q}{\ensuremath{\mathbb Q}}
\newcommand{\N}{\ensuremath{\mathbb N}}
\newcommand{\Z}{\ensuremath{\mathbb Z}}
\newcommand{\R}{\ensuremath{\mathbb R}}
\newcommand{\K}{\ensuremath{\mathbb K}}


\newcommand{\ds}{\ensuremath{\displaystyle}}
\renewcommand{\qedsymbol}{$\dashv$}


\newcommand{\ca}[1]{\ensuremath{\mathcal{#1}}}
\newcommand{\set}[2]{\ensuremath{\{#1 \ \mid \ #2\}}}
\newcommand{\arr}[1]{\ensuremath{\overrightarrow{#1}}}
\newcommand{\barr}[1]{\ensuremath{\overline{#1}}}



\newcommand{\tu}[1]{\textup{#1}}
\newcommand{\param}{\ensuremath{\mathit{O}}}
\newcommand{\seq}{\ensuremath{\mathtt{Seq}}}
\newcommand{\norm}[2]{\ensuremath{\|#1-#2\|}}
\newcommand{\distf}{\ensuremath{\textrm{Df}(\om)}}
\newcommand{\completion}[1]{\ensuremath{[(\om,#1)]}}


\newtheorem{theorem}{Theorem}[section]
\newtheorem{lemma}[theorem]{Lemma}
\newtheorem{definition}[theorem]{Definition}
\newtheorem{proposition}[theorem]{Proposition}
\newtheorem{corollary}[theorem]{Corollary}
\newtheorem{remark}[theorem]{Remark}
\newtheorem{remarks}[theorem]{Remarks}
\newtheorem{example}[theorem]{Example}
\newtheorem{examples}[theorem]{Examples}

\title{Turning Borel sets into Clopen sets effectively}

\author[V. Gregoriades]{Vassilios Gregoriades}

\address{Technische Universit\"{a}t Darmstadt, Fachbereich Mathematik,
Arbeitsgruppe Logik, Schlo{\ss}gartenstra{\ss}e 7,
64289, Darmstadt, Germany.}
\email{gregoriades [at] mathematik [dot] tu-darmstadt [dot] de}

\thanks{The first half of this article (roughly until Remark \ref{remark the parameter may not be in del}) are included in the Ph.D. Thesis of the author, which was submitted and approved in 2009 by the University of Athens, Greece. The author would like to thank his supervisor \textsc{Yiannis Moschovakis} for his motivating ideas and his invaluable guidance. The author is currently a post-doctoral researcher at TU Darmstadt, Germany in the workgroup of \textsc{Ulrich Kohlenbach}, whom the author would like to thank for his substantial support.}

\subjclass[2010]{Primary 03E15, 54H05, 03D55.}
\keywords{Turning Borel sets into clopen sets, extension of a Polish topology, hyperarithmetical points, uniformity results.}

\date{\today}

\maketitle

\begin{abstract}
We present the effective version of the theorem about turning Borel sets in Polish spaces into clopen sets while preserving the Borel structure of the underlying space. We show that under some conditions the emerging parameters can be chosen in a hyperarithmetical way and using this we prove a uniformity result.
\end{abstract}

\section{Introduction.}

One of the topics of effective descriptive set theory is the refinement of well-known theorems in recursive theoretic terms. This ``effective" version of a theorem is stronger than the original one and as it is often the case it provides a uniformity result which does not seem to follow from the original statement. A typical example is \emph{Suslin's Theorem}, which states that every bi-analytic subset of a Polish space is Borel, and its refinement the \emph{Suslin-Kleene Theorem}, which provides a recursive (and thus continuous) function $u: \ca{N} \times \ca{N} \to \ca{N}$, where $\ca{N} = \om^\om$, such that whenever $\alpha$ and $\beta$ are codes of complementary analytic sets, say $A$ and $\ca{X} \setminus A$, then $u(\alpha, \beta)$ is a Borel code of $A$, c.f. 7B.4 in \cite{yiannis_dst} and \cite{yiannis_classical_dst_as_a_refinement_of_effective_dst}.

In this article we prove the effective version of the following well-known theorem of classical descriptive set theory: if $(\ca{X},\ca{T})$ is a Polish space and $A$ is a Borel subset of \ca{X}, then $(*)$ there is a Polish topology $\ca{T}_{\infty}$ on \ca{X}, which extends \ca{T}, yields the same Borel sets as \ca{T}, and $A$ is clopen in $(\ca{X},\ca{T}_{\infty})$, c.f. Theorem \ref{theorem main}. \vg{Let us see how the effective version of this theorem should be formulated. First we start with a recursively presented Polish space (see below for the definitions) and we replace the term ``Borel" with ``\del". The arising space $(\ca{X},\ca{T}_{\infty})$ may not be recursively presented but we can always relativize with respect to some parameter $\ep \in \ca{N } = \om^\om$. So we will ask for $(\ca{X},\ca{T}_{\infty})$ to be an $\ep$-recursively presented Polish space. Then we replace the terms ``clopen" with ``$\Delta^0_1(\ep)$" and ``Borel in $\ca{T}_{\infty}$" with ``$\del(\ep)$". Summarizing the effective version should read as follows. Suppose that $(\ca{X},d)$ is a recursively presented Polish space and $A$ is a \del \ subset of \ca{X}. Then there is some $\ep \in \ca{N}$ and a Polish topology $\ca{T}_{\infty}$ which extends \ca{T} such that: (1) the space $(\ca{X},\ca{T}_{\infty})$ is $\ep$-recursively presented; (2) the set $A$ is a $\Delta^0_1(\ep)$ subset of $(\ca{X},\ca{T}_{\infty})$ and (3) any $B \subseteq \ca{X}$ is a $\del(\ep,\alpha)$ subset of $(\ca{X},\ca{T})$ exactly when $B$ is a $\del(\ep,\alpha)$ subset of $(\ca{X},\ca{T}_{\infty})$ for any parameter $\alpha$, c.f. Theorem \ref{theorem main}. A natural question which arises is to determine the best such parameter \ep \ in terms of definability. We will see that the latter \ep \ can be chosen to be recursive in Kleene's \param \ and we will spend an entire section of this article dealing with the problem of choosing a hyperarithmetical such \ep.}

Our proof will not follow the same way as the usual proof of the latter theorem. We will instead present a different proof of this result, which is less well-known (and perhaps new). We will see that there are certain advantages with our approach which allow us to proceed with the effective version. Let us give a very brief description of the usual proof. One starts with a Polish space \ca{X}, defines $\ca{S} = \set{A \subseteq \ca{X}}{A \ \textrm{is Borel and satisfies $(*)$ above}}$ and shows that \ca{S} is a $\sigma$-algebra which contains the open sets and so it contains every Borel set. The effective version of this proof requires the notion of the effective $\sigma$-field (c.f. \cite{yiannis_dst} Section 7B) and a rather messy encoding of Polish topologies as we proceed with the induction. Moreover this approach seems to provide little information on the best choice for the parameters which shall emerge.

On the other hand our approach is based on the following result of Lusin-Suslin: every Borel subset of a Polish space is the injective continuous image of closed subset of \ca{N}, c.f. 13.7 \cite{kechris_classical_dst}.  So let us assume that \ca{X} is a Polish space and that $A$ is a Borel subset of \ca{X}. Then there are closed sets $F_1, F_2 \subseteq \ca{N}$ and continuous functions $\pi_1, \pi_2 : \ca{N} \to \ca{X}$ such that $\pi_i$ is injective on $F_i$, $i=1,2$, $\pi_1[F_1] = A$ and $\pi_2[F_1] = \ca{X} \setminus A$. Now we define a distance function $d_A$ on $A$ in such a way that $\pi_1$ becomes an isometry i.e., $d_A(x,y) = \pi_1(\alpha,\beta)$, where $\alpha,\beta \in F_1$ with $\pi_1(\alpha) = x$ and $\pi_1(\beta) = y$. Similarly we define the distance function $d_{A^c}$ on the complement of $A$ and then we consider the direct sum $(A,d_A) \oplus (\ca{X} \setminus A, d_{A^c})$. The latter space has all required properties.\footnote{One may express some reservations on whether our suggested proof is indeed independent from the original one, because one can derive the Lusin-Suslin Theorem from the theorem about turning Borel sets into clopen, c.f. \cite{kechris_classical_dst}. There are however straightforward proofs of the Lusin-Suslin Theorem, see for example 1G.5 or -for a very different proof- 4A.7 in \cite{yiannis_dst}. Therefore the proof that we are now suggesting is indeed independent from the usual proof. Nevertheless as it becomes clear from our comments these two theorems are ``equivalent" in the sense that we can prove one from the other.}

The advantage of the latter proof is that it reduces the problem from Borel sets to closed sets where the verification of Cauchy-completeness is obvious. Moreover it is straightforward to effectivize, for the effective version of the Luslin-Suslin Theorem has been proved by Moschovakis c.f. 4A.7 \cite{yiannis_dst}. \vg{Therefore the problem is reduced into finding a recursive presentation inside $\Pi^0_1$ sets, which can be done recursively in Kleene's \param. This method will also give us some sufficient conditions in order to choose the recursive presentation in a \del \ way.} Finally, as we will see, this approach provides very clear information about the emerging parameters.

In the rest of this section we recall the basic definitions and notations. We assume that the reader is familiar with recursion theory and effective descriptive set theory, c.f. \cite{yiannis_dst} Chapter 3. In the next section we prove our main theorem, c.f. Theorem \ref{theorem main}, and afterwards we examine the problem of choosing the emerging parameters in a \del \ way, c.f. Theorems \ref{theorem countable sets admit good parameter in del} and \ref{theorem G_delta and complement G_delta}. We conclude this article with a related uniformity result about choosing the extended topology $\ca{T}_{\infty}$ in a Borel way, c.f. Theorem \ref{theorem uniformity result}. This is the only result of this article whose statement is purely classical in the sense that it involves no notions from effective theory. Its proof however is a corollary of almost all preceding effective results.

\textbf{Notation and Definitions.} By \om \ we mean the least infinite ordinal, which we identify with the set of natural numbers, and by $\om_1$ the least uncountable ordinal. We fix once and for all a recursive encoding $\langle \cdot \rangle$ of all finite sequences of naturals by a natural number. Number $0$ will be the code of the empty sequence. We denote by \seq \ the recursive set of all codes of finite sequences in \om. If $s \in \seq$ is a code of $u = (u_0,\dots,u_{n-1})$ we define $lh(s)=n$ and if $i < n$ we put $(s)_i = u_i$. If $s \in \seq$ and $i \geq lh(s)$ or if $s \not \in \seq$ and $i$ is arbitrary we define $(s)_i$ to be $0$. Finally we fix the following enumeration of the rational numbers: $q_s = (-1)^{(s)_0}\frac{(s)_1}{(s)_2 + 1}$ for $s \in \om$.

We denote by \ca{N} the space $\om^\om$ of all infinite sequences of naturals with the product topology. The space \ca{N} is the \emph{Baire space}. The members of \ca{N} will be denoted by lowercase Greek letters such as $\alpha, \beta$ etc. We fix the usual distance function $p_\ca{N}$ on $\ca{N}$, which is defined as follows: $p_\ca{N}(\alpha,\beta) = (\textrm{least} \ k \ \alpha(k) \neq \beta(k) + 1)^{-1}$ for $\alpha \neq \beta$. We may view every $\alpha \in \ca{N}$ as a code of an infinite sequence in \ca{N}. To be more specific, we define $(\alpha)_i(n) = \alpha(\langle i,n \rangle)$ for all $i \in \om$ and so $\alpha$ gives rise to the sequence $((\alpha)_i)_{\iin}$. Of course we may apply the inverse procedure: if $(\alpha_i)_{\iin}$ is a sequence in \ca{N} there is some $\alpha \in \ca{N}$ such that $(\alpha)_i = \alpha_i$ for all \iin. For $\alpha, \beta \in \ca{N}$ we will denote by $\langle \alpha, \beta \rangle$ the unique $\gamma \in \ca{N}$ such that $(\gamma)_0 = \alpha$, $(\gamma)_1 = \beta$ and $\gamma(t) = 0$ if $t \neq \langle i,n \rangle$ for all $i=0,1$ and all \n.

We will often identify relations with the sets that they define and write $P(x)$ instead of $x \in P$.

A topological space is a \emph{Polish space} if it is separable and is metrizable by a complete distance function. We will call such a distance function as a \emph{suitable} distance function. We employ the standard hierarchy $(\tboldsymbol{\Sigma}^0_\xi)_{\xi < \om_1}$ of Borel sets in Polish spaces c.f. \cite{kechris_classical_dst}. Suppose now that $(\ca{X},d)$ is a complete and separable metric space. A sequence $(x_n)_{\n}$ is a \emph{recursive presentation of $(\ca{X},d)$} if (1) the set \set{x_n}{\n} is dense in \ca{X} and (2) the relations of $\om^4$ defined by $P_{<}(i,j,k,m) \eq d(x_i,x_j) < \frac{k}{m+1}$ and $P_{\leq} (i,j,k,m) \eq d(x_i,x_j) \leq \frac{k}{m+1}$ are recursive. We say that $(\ca{X},d)$ \emph{admits a recursive presentation} or that \emph{$(\ca{X},d)$ is recursively presented} if there is a sequence $(x_n)_{\n}$ in \ca{X} which satisfies the previous properties (1) and (2). For every complete space $(\ca{X},d)$ with a recursive presentation $(x_n)_{\n}$ we consider the set $B(x_n,m,k) = \set{x \in \ca{X}}{d(x,x_n) < \frac{m}{k+1}}$. The latter is either the empty set or a ball with center $x_n$ and radius $\frac{m}{k+1}$. For $s \in \om$ define $N(\ca{X},s) = B(x_{(s)_0},(s)_1,(s)_2)$. The family \set{N(\ca{X},s)}{s \in \om} is the associated \emph{neighborhood system of \ca{X}} (with respect to $(x_n)_{\n}$ and $d$) and it is clear that it forms a basis for the topology of \ca{X}. We say that a Polish space \ca{X} is recursively presented if there is a suitable distance function $d$ such that the corresponding space $(\ca{X},d)$ is recursively presented. When we refer to a recursively presented Polish space we will always assume that we are given the suitable distance function and the recursive presentation. Standard examples of recursively presented Polish spaces are the Baire space \ca{N}, the real numbers and \om.

The property of being recursively presented is clearly carried out to finite products and sums of spaces i.e., if $\ca{X}_1, \dots, \ca{X}_n$ are recursively presented Polish spaces then both $\ca{X}_1 \times \dots \times \ca{X}_n$ and $\ca{X}_1 \oplus \dots \oplus \ca{X}_n$ are recursively presented. We fix once and for all a scheme for passing from the recursive presentations of finitely many spaces to recursive presentations of their sum and product.

Suppose that $(\ca{X},d)$ is complete and recursively presented. We say that $A \subseteq \ca{X}$ \emph{is in $\Sigma^0_1$} (or that \emph{$A$ is a $\Sigma^0_1$ set}, or that \emph{$A$ is semirecursive}) if there is a recursive function $f: \om \to \om$ such that $A = \cup_{s \in \om} N(\ca{X},f(s))$. In other words $\Sigma^0_1$ sets are the unions of a recursive collection from the family of our fixed open neighborhoods. This definition suggests that the property of being a $\Sigma^0_1$ set depends on the way we have encoded the basic neighborhoods $N(\ca{X},s)$; however this is not the case c.f. 3C.12 \cite{yiannis_dst}. (It does depend of course on the distance function and the recursive presentation.) The set $A$ is in $\Pi^0_1$ if $\ca{X} \setminus A$ is in $\Sigma^0_1$. Inductively we define the family of $\Sigma^0_{n+1}$ subsets of \ca{X} as the family of all sets which are the projection along \om \ of a $\Pi^0_n$ subset of $\ca{X} \times \om$ and $\Pi^0_{n+1}$ sets as the complements of sets in $\Sigma^0_{n+1}$. The set $A$ is in \sig \ if it is the projection along \ca{N} of a $\Pi^0_1$ subset of $\ca{X} \times \ca{N}$. The set $A$ is in \pii \ if its complement is in \sig \ and $A$ is in \del \ if it is both in \sig \ and \pii. Of course all these definitions coincide with the usual ones of Kleene in the case where $\ca{X} = \om$ or $\ca{X} = \ca{N}$.

A function $f: \ca{X} \to \ca{Y}$ between recursively presented Polish spaces is \emph{recursive} (or \emph{$\Sigma^0_1$-recursive}) if the relation $R^f \subseteq \ca{X} \times \om$ defined by $R^f(x,s) \eq f(x) \in N(\ca{Y},s)$ is in $\Sigma^0_1$. Similarly the function $f$ is \emph{\del-recursive} if the previous set $R^f$ is in \del. A point $x \in \ca{X}$ \emph{is in \del} or \emph{it is a \del \ point} exactly when the relation $U \subseteq \om$ defined by $U(s) \eq x \in N(\ca{X},s)$, is in \del.

A very important notion is the one of relativization with respect to some parameter. Suppose that \ca{X} and \ca{Y} are recursively presented Polish spaces and that $y \in \ca{Y}$. A subset $A$ of \ca{X} is in $\Sigma^0_1(y)$ if there is some $P \subseteq \ca{Y} \times \ca{X}$ such that $A$ is the $y$-section of $P$ i.e., $A = P_y:=\set{x \in \ca{X}}{P(x,y)}$. Similarly one defines the classes of sets $\Sigma^0_n(y)$, $\Pi^0_n(y)$, $\sig(y)$, $\pii(y)$, $\del(y)$, the $\del(y)$-recursive functions and the $\del(y)$ points. If a point $x$ is in $\del(y)$ we will also say that \emph{$x$ is hyperarithmetical in $y$}. The relativization applies to recursive presentations as well. Consider a point $\ep \in \ca{N}$. A sequence $(x_n)_{\n}$ in a complete metric space $(\ca{X},d)$ is an \emph{\ep-recursive presentation of $(\ca{X},d)$} if the previous conditions (1) and (2) are satisfied with the modification that the relations $P_<$ and $P_\leq$ are now $\ep$-recursive. One can then repeat all previous definitions by replacing everywhere the term ``recursive" with ``$\ep$-recursive". For example a subset $A$ of an $\ep$-recursively presented Polish space \ca{X} is \emph{$\ep$-semirecursive} if there is an $\ep$-recursive function $f: \om \to \om$ such that $A = \cup_{s \in \om}N(\ca{X},f(s))$. We will denote the class of $\ep$-semirecursive sets with $\Sigma^0_1(\ep)$. There is a potential double meaning for $\Sigma^0_1(\ep)$, because every recursive presentation is also an $\ep$-recursive presentation for any $\ep \in \ca{N}$. One can check however that the $\ep$-recursive unions of basic neighborhoods $N(\ca{X},s)$ -where \ca{X} is recursively presented- are exactly the $\ep$-sections of $\Sigma^0_1$ subsets of $\ca{N} \times \ca{X}$. So no conflict arises.

It is not true that every Polish space is recursively presented but it is easy to see that every Polish space is recursively presented in some parameter $\ep$. All theorems about recursively presented Polish spaces are transferred to $\ep$-recursively presented Polish spaces. The latter claim is the \emph{Relativization Principle}. This principal is fundamental for the applications of effective descriptive set theory to ``classical mathematics" such as Theorem \ref{theorem uniformity result}. For more information refer to 3I of \cite{yiannis_dst}.

A function $f: (\ca{X},d_{\ca{X}}) \to (\ca{Y},d_{\ca{Y}})$ is an \emph{isometry} if $d_{\ca{Y}}(f(x),f(y)) = d_{\ca{X}}(x,y)$ for all $x,y \in \ca{X}$ and $f$ is surjective. If the spaces $(\ca{X},d_{\ca{Y}})$ and $(\ca{Y},d_{\ca{Y}})$ are complete and recursively presented it is not hard to verify that a set $A$ is in $\Gamma$ exactly when $f[A]$ is in $\Gamma$, where $\Gamma$ is any of the classes $\Sigma^0_n(\alpha)$, $\Pi^0_n(\alpha)$, $\sig(\alpha)$, $\pii(\alpha)$ and $\del(\alpha)$.

We are going to use some fundamental theorems of effective descriptive set theory including (but not restricted to) the Theorem on Restricted Quantification (4D.3), the Effective Perfect Set Theorem (4F.1), the Strong $\Delta$-Selection Principle (4D.6) and the theorem about the existence of \del \ points inside \pii \ non-meager sets (4F.20). (All previous references are from \cite{yiannis_dst}). We should mention explicitly the following result.

\begin{theorem}{\normalfont (c.f. 4A.7 \cite{yiannis_dst})}
\label{theorem yiannis 4A7}
Every \del \ subset of a recursively presented Polish space is the recursive injective image of $\Pi^0_1$ subset of \ca{N}.
\end{theorem}

\vg{\begin{theorem}{\normalfont(Theorem on Restricted Quantification, c.f. 4D.3 \cite{yiannis_dst} and \cite{kleene_quantification_numbertheoretic_functions})}
\label{theorem on restricted quantification}
Suppose that \ca{X} and \ca{Y} are recursively presented Polish spaces and that $Q
\subseteq \ca{X} \times \ca{Y}$ is in \pii. The the set $P \subseteq \ca{X}$ defined by
\[
P(x) \eq (\exists y \in \del(x))Q(x,y)
\]
is also in \pii.
\end{theorem}

\begin{theorem}{ \normalfont (The Effective Perfect Set Theorem, c.f. 4F.1 \cite{yiannis_dst} and \cite{harrison_phd_thesis})}
\label{theorem effective perfect set theorem}
Suppose that \ca{X} is a recursively presented Polish space and that $P$ is a \sig \ subset of \ca{X}. If $P$ has at least one member
not in \del \ then $P$ contains a perfect subset.
\end{theorem}

\begin{theorem}{\normalfont (c.f. 4F.20 \cite{yiannis_dst}, \cite{thomason_the_forcing_method_and_the_upper_lattice_of_hyperdegrees}, \cite{hinman_some_applications_of_forcing_to_hierarchy_problems_in_arithmetic}, \cite{kechris_measure_and_category_in_effective_descriptive_set_theory})}
\label{theorem non meager pii contains a member in del}
If \ca{X} is a recursively presented Polish space and $P$ is a \pii \ subset of \ca{X}, which is non-meager, then $P$ contains a member in \del.
\end{theorem}
}

\section{The effective version.}

In this section we present our main theorem which is the effective version of the theorem about turning a Borel subset of a Polish space into a clopen set.

\begin{theorem}

\label{theorem main}

Suppose that $(\ca{X},\ca{T})$ is a recursively presented Polish space, $d$ is a suitable distance function for $(\ca{X},\ca{T})$ and $A$ is a \del \ subset of \ca{X}. Then there exist an $\ep_A \in \ca{N}$, which is recursive in \param \ and a Polish topology $\ca{T}_{\infty}$ with suitable distance function $d_{\infty}$, which extends $\ca{T}$ and has the following properties: \tu{(1)} the Polish space $(\ca{X},\ca{T}_{\infty})$ is $\ep_A$-recursively presented, \tu{(2)} the set $A$ is a $\Delta^0_1(\ep_A)$ subset of
$(\ca{X},d_{\infty})$, \tu{(3)} if $B \subseteq \ca{X}$ is a $\del(\alpha)$ subset
of $(\ca{X},d)$, where $\alpha \in \ca{N}$, then $B$ is a $\del(\ep_A,\alpha)$ subset of
$(\ca{X},d_{\infty})$ and \tu{(4)} if $B \subseteq \ca{X}$ is a $\del(\ep_A,\alpha)$ subset of $(\ca{X},d_{\infty})$, where $\alpha \in \ca{N}$, then $B$ is a $\del(\ep_A,\alpha)$ subset of $(\ca{X},d)$.
\end{theorem}

\begin{proof}
From Theorem \ref{theorem yiannis 4A7} there are $\Pi^0_1$ sets $F_1, F_2 \subseteq \ca{N}$ and recursive functions $\pi_1, \pi_2 : \ca{N} \to \ca{X}$ such that $\pi_i$ is one-to-one on $F_i$, $i=1,2$ and $A = \pi_1[F_1]$, $A^c =
\pi_2[F_2]$; where $A^c$ stands for the complement of $A$ in
\ca{X}. We may assume that $A \neq \emptyset, \ca{X}$ for
otherwise the result is immediate. So $F_1, F_2 \neq \emptyset$.

Now we define a distance function on $A$ such that the
function $\pi_1$ becomes an isometry. Let $p_{\ca{N}}$
be the usual distance function on the Baire space \ca{N}, $x,y \in
A$ and $\alpha, \beta$ be the unique members of $F_1$ such that
$x= \pi_1(\alpha), y= \pi_1(\beta)$. Put
$$
d_A(x,y) = d_A(\pi_1(\alpha),\pi_1(\beta)) =
p_{\ca{N}}(\alpha,\beta)
$$
Similarly one defines the distance function $d_{A^c}$ on the complement of $A$ so that $\pi_2$ becomes an isometry. Let $p_{F_i}$ be the restriction of the distance function $p_{\ca{N}}$ on $F_i \times F_i$, $i=1,2$. Since the sets $F_i$ are closed the metric spaces $(F_i,p_{F_i})$ are
separable and complete, $i=1,2$. Also since the functions $\pi_i
\upharpoonright F_i$ are isometries we have that the spaces
$(A,d_A)$, $(A^c,d_{A^c})$ are separable and complete as well.

Now we are going to define $\ep_1, \ep_2 \in \ca{N}$ in which
the metric spaces $(A,d_A)$, $(A^c,d_{A^c})$ admit a recursive
presentation respectively. Recall the set $\seq \subseteq \om$ of the codes of all finite sequences of \om. For $s \in \seq$ we define
\[
N_s =\set{\alpha \in \ca{N}}{\alpha(i) = (s)_i \ \forall i < lh(s)}.
\]
We adopt the notation $s \ \hat{} \ k$ for the natural number $\langle (s)_0, \dots, (s)_{lh(s)-1}, k \rangle$, where $s \in \seq$ and $k \in \om$. Define
\[
\ep_1(s) = 1 \eq \ \seq(s) \ \& \ F_1 \cap N_s \neq
\emptyset
\]
and $\ep_1(s)=0$ otherwise. Notice that for all $s$
with $\ep_1(s)=1$ there always exists some $k \in \om$ such that
$\ep_1(s \ \hat{} \ k) =1$.

We now define a sequence $(\alpha_s)_{s \in \om}$ which is dense in $F_1$. First for $\ep_1(s) = 1$ define $\alpha_s(i)
= (s)_i$ for all $i < lh(s)$ and for $i \geq lh(s)$
$$
\alpha_s(i) = (\mu k)[\ep_1(\langle \alpha_s(0), \dots,
\alpha_s(i-1),k \rangle ) = 1]
$$
It is clear that $\alpha_s \in N_s$ and that $N_{\langle \alpha_s(0), \dots,
\alpha_s(i-1) \rangle} \cap F_1 \neq \emptyset$ for all $s$ with $\ep_1(s) = 1$ and for all \iin. Since $F_1$ is closed it follows that $\alpha_s \in N_s \cap F_1$ for all $s \in \seq$ such that $N_s \cap F_1 \neq \emptyset$.
Therefore the sequence $(\alpha_s)_{ \{ s : \ \ep_1(s)=1\} }$ is
dense in $F_1$.  Now pick the least natural $s$, call it $s_0$, for which $\ep_1(s) = 1$ and define $\alpha_s = \alpha_{s_0}$ for all $s$, for which $\ep_1(s)=0$.

We will prove that the sequence $(\alpha_s)_{s \in \om}$ is an
$\ep_1$-recursive presentation of $(F_1,p_{F_1})$. First notice
that the relation $P \subseteq \om^3$ defined by $P(s,t,i) \eq \ \alpha_s(i) = \alpha_t(i)$, is recursive in $\ep_1$. To see this notice that the definitions
above can be reformulated so that $s$ becomes a variable i.e., there is
a $\Sigma^0_1(\ep_1)$-recursive function $f : \om \times \om \to
\om$ such that for all $s,i$ we have that $f(s,i) = \alpha_s(i)$.

Now we claim that the relations $Q \subseteq \om \times \om$ and $R \subseteq \om \times \om$ defined by
\begin{eqnarray*}
Q(s,t,i) &\eq& P(s,t,i) \ \& \ (\forall j<i)\neg P(s,t,j)\\
&\eq& i \ \emph{is the least $k$ for which $\alpha_s(k) \neq \alpha_t(k)$}
\end{eqnarray*}
and $R(s,t) \eq \alpha_s = \alpha_t$, are recursive in $\ep_1$. This is clear for $Q$, since $P$ is recursive in $\ep_1$. We prove the claim for $R$. If $\ep_1(s)=\ep_1(t) =1$ then the equality between $\alpha_s$ and $\alpha_t$ can be verified by only finite values; in particular
\begin{eqnarray*}
\alpha_s = \alpha_t &\eq& \ \ \ [ \ s \sqsubseteq t \ \& \
(\forall i)[ lh(s) \leq i < lh(t)] \rightarrow
\alpha_s(i)=\alpha_t(i) \ ]
\\
&& \vee \ [ \ t \sqsubseteq s \ \& \ (\forall i)[ lh(t) \leq i <
lh(s)] \rightarrow \alpha_s(i)=\alpha_t(i) \ ]
\end{eqnarray*}
Similarly if $\ep_1(s) =1$ and $\ep_1(t) = 0$ then $\alpha_s = \alpha_t \eq \alpha_s = \alpha_{s_0}$. If $\ep_1(s) = \ep_1(t) = 0$ then clearly $\alpha_s  = \alpha_{s_0} = \alpha_t$. From these remarks it follows that the relation $R$ is recursive in $\ep_1$. From this it is easy to check that the relation $P_< \subseteq \om^4$ defined by $P_<(s,t,m,k) \eq p_{F1}(\alpha_s,\alpha_t) < \frac{m}{k+1}$, is recursive in $\ep_1$ and similarly for the corresponding relation $P_{\leq}$.\vg{We now compute
\begin{eqnarray*}
P_<(s,t,m,k) &\eq& p_{F1}(\alpha_s,\alpha_t) < \frac{m}{k+1}
\\
&\eq& [ \ R(s,t) \ \& \ m>0 \ ]
\\
&& \ \vee \ (\exists i)[ \ \alpha_s(i) \neq \alpha_t(i) \ \& \
(\forall j<i)[\alpha_s(j)=\alpha_t(j)]
\\
&& \ \& \ k+1 < (i+1) \cdot m \ ]
\\
&\eq& [ \ R(s,t) \ \& \ m>0 \ ]
\\
&& \ \vee (\exists i)[ \ Q(s,t,i) \ \& \ k+1 < (i+1) \cdot m \ ]
\\
&\eq& [ \ R(s,t) \ \& \ m>0 \ ]
\\
&& \ \vee (\forall i)[ \ Q(s,t,i) \rightarrow  k+1 < (i+1) \cdot m \ ]
\end{eqnarray*}
Thus the relation $P_<$ is recursive in $\ep_1$. It clear that the relation $P_\leq$ is recursive in $\ep_1$ as well.}Therefore the metric space $(F_1,p_{F_1})$ is recursively
presented in $\ep_1$. Since $\pi_1$ is an isometry between $(F_1,p_{F_1})$ and $(A,d_A)$ it follows that the sequence $(\pi_1(\alpha_s))_{s \in \om}$ is an $\ep_1$-recursive presentation of $(A,d_A)$. Similarly we define $\ep_2$ for the metric space $(A^c,d_{A^c})$ with the $\ep_2$-recursive presentation $(\pi_1(\alpha_s))_{s \in \om}$. We define $\ep_A = \langle \ep_1, \ep_2 \rangle$ and we check that $\ep_A$ is recursive in \param. To see this consider the \sig \ relations $P_1(s) \eq F_1 \cap N(\ca{X},s) \neq \emptyset \ \ \textrm{and} \ \ P_2(s) \eq F_2 \cap N(\ca{X},s) \neq \emptyset$, and notice that $\ep_A(\langle i,s \rangle) = 1 \eq [ \ i=0 \ \& \ P_1(s) \ ] \ \vee \ [ \ i=1 \ \& \ P_2(s) =1 \ ]$.
\vg{\begin{eqnarray*}
\ep_A(\langle i,s \rangle) = 1 &\eq& [ \ i=0 \ \& \ \ep_1(s) =1 \
] \ \vee \ [ \ i=1 \ \& \ \ep_2(s) =1 \ ]
\\
&\eq& [ \ i=0 \ \& \ P_1(s) \ ] \ \vee \ [ \ i=1 \ \& \ P_2(s) =1
\ ]
\end{eqnarray*}
}
Thus $\ep_A$ is recursive in a \sig \ subset of \om. Since \param \ is a complete \pii \ set it follows that $\ep_A$ is recursive in \param.

We now define the topology $\ca{T}_{\infty}$ on \ca{X} as the
direct sum of $(A,d_A)$ and $(A^c,d_{A^c})$ i.e., $d(x,y) = 2$ if it is not the case $x,y \in A$ or $x,y \in A^c$ and otherwise $d$ coincides with $d_A$ or $d_{A^c}$. As we mentioned in the introduction $(\ca{X},d_{\infty})$ is Polish and recursively presented in $\ep_A$. A suitable dense sequence is $x_{\langle 0,s \rangle} = \pi_1(\alpha_s)$, $x_{\langle 1,s \rangle} = \pi_2(\beta_s)$ and $x_t = \alpha_{s_0}$ in any other case. It is clear that $V$ is $d_{\infty}$-open if and only if $V \cap A$ is $d_A$-open and $V \cap A^c$ is $d_{A^c}$-open for all $V \subseteq \ca{X}$.

The set $A$ is $\Delta^0_1(\ep_A)$ in $(\ca{X},d_{\infty})$. We prove that the topology $\mathcal{T}_{\infty}$ which is generated by $d_{\infty}$ extends $\mathcal{T}$. Let $V$ be in \ca{T}. We will show that $V \cap A$ is $d_A$-open. The proof for $V \cap A^c$ is similar. Let $y = \pi_1(\alpha) \in V \cap A$, with $\alpha \in F_1$. Since $\pi_1: \ca{N} \to (\ca{X},d)$ is continuous there is some $k
\in \om$ such that for all $\beta \in \ca{N}$ with
$p_{\ca{N}}(\alpha,\beta) < \frac{1}{k+1}$ we have that $\pi_1(\beta) \in
V$. We claim that the $d_A$-ball of center $y$ and radius
$\frac{1}{k+1}$ is contained in $V$. Let $z = \pi_1(\beta)$, with $\beta \in F_1$ be such that $d_A(y,z) < \frac{1}{k+1}$. Then $p_{\ca{N}}(\alpha,\beta) =
d_A(\pi_1(\alpha),\pi_1(\beta)) = d_A(y,z) < \frac{1}{k+1}$, hence $z = \pi_1(\beta) \in V$.

Now we deal with the rest of the assertions. It is not hard to see that every recursive function $\pi: \ca{N} \to (\ca{X},d)$ is $\del(\ep_A)$-recursive as a function from $\ca{N}$ to $(\ca{X},d_{\infty})$ (by considering the standard recursive presentation of \ca{N} as an $\ep_A$-recursive presentation).\vg{To see this consider the relation $R^{\pi}$ defined by $R^{\pi}(\alpha,t) \eq \pi(\alpha) \in N((\ca{X},d_{\infty}),t)$. We compute
\begin{eqnarray*}
&&\hspace{-1cm}d_{\infty}(\pi(\alpha),x_{\langle i,s \rangle}) < \frac{k}{m+1} \eq \\
&&  \ \ \ \ [ \ \pi(\alpha) \in A \ \& \ i=0 \ \& \ d_A(\pi(\alpha),\pi_1(\alpha_s)) < \frac{k}{m+1} \ ]\\
&& \ \vee \ [ \ \pi(\alpha) \in A^c \ \& \ i=1 \ \& \ d_{A^c}(\pi(\alpha),\pi_2(\beta_s)) < \frac{k}{m+1} \ ]\\
&& \ \vee \ [ \ [[\pi(\alpha) \in A \ \& \ i=1] \vee [\pi(\alpha) \in A^c \ \& \ i=0] ] \ \& \ 1 < \frac{k}{m+1} \ ].
\end{eqnarray*}
}
It follows from 4A.7 and 4D.7 of \cite{yiannis_dst} that every \del \ subset of $(\ca{X},d)$ is a $\del(\ep_A)$ subset of $(\ca{X},d_\infty)$. This settles assertion (3).

About (4) let $B$ be a $\del(\ep_A)$ subset of
$(\ca{X},d_{\infty})$. The set $\pi_1^{-1}[B \cap A]
\subseteq F_1 \subseteq \ca{N}$ is in $\del(\ep_A)$ since
$\pi_1$ is recursive as a
function from \ca{N} to $(\ca{X},d)$ and therefore $\del(\ep_A)$-recursive as a function from \ca{N} to $(\ca{X},d_{\infty})$. Since $\pi_1$ is
one-to-one on $\pi_1^{-1}[B \cap A]$ again from 4D.7 of
\cite{yiannis_dst} we have that $B \cap A = \pi_1[\pi_1^{-1}[B \cap A]]$
is a $\del(\ep_A)$ subset of $(\ca{X},d)$. Similarly
we prove that the set $B \cap A^c$ is a $\del(\ep_A)$
subset of $(\ca{X},d)$ as well. Therefore the set $B = (B
\cap A) \cup (B \cap A^c)$ is a $\del(\ep_A)$ subset of
$(\ca{X},d)$.
\end{proof}

\section{Computing the parameter.}

As we have seen one can choose the parameter \ep \ of Theorem \ref{theorem main} to be recursive in \param. Since we are dealing mostly with \del \ sets a natural question to ask is whether we can choose a hyperarithmetical such \ep. We will see that the latter is not necessarily true even if we start with a $\Sigma^0_1$ set which we want to turn into $\Delta^0_1$. Nevertheless we will show that in some cases it is possible to choose a hyperarithmetical such \ep, c.f. Theorem \ref{theorem countable sets admit good parameter in del}. Moreover we will see that, under some specific assumptions about the set that we start with and our underlying space, we can characterize the case where the choice of a hyperarithmetical such \ep is possible, c.f. Theorem \ref{theorem G_delta and complement G_delta}.

\begin{definition}

\normalfont

\label{definition suitable parameter}

Suppose that $(\ca{X},\ca{T})$ is a recursively presented Polish space, $d$ is a suitable distance function and that $A$ is a \del \ subset of \ca{X}.

(1) We say that $\ep \in \ca{N}$ is a \emph{good parameter for $A$} if all conclusions of Theorem \ref{theorem main} are satisfied by taking $\ep_A$ as $\ep$. The latter means that there is a Polish topology $\ca{T}_{\infty}$ on \ca{X} which extends \ca{T}; the space $(\ca{X},\ca{T}_{\infty})$ is $\ep$-recursively presented; the set $A$ is $\Delta^0_1$ in $(\ca{X},d_{\infty})$, where $d_{\infty}$ is a suitable distance function; every $B$ which is a $\del(\alpha)$ subset of $(\ca{X},d)$ is a $\del(\ep,\alpha)$ subset of $(\ca{X},d_{\infty})$ and every $B$ which is a $\del(\ep,\alpha)$ subset of $(\ca{X},d_{\infty})$ is a $\del(\alpha,\ep)$ subset of $(\ca{X})$.

(2) We say that \emph{the class \del \ is dense in $A$} if whenever some $V \in \ca{T}$ intersects $A$ then $V$ intersects $A$ in \del \ point of $(\ca{X},d)$. \end{definition}

The next proposition gives a necessary condition for a set in \del \ in order to admit a good parameter in \del.

\begin{proposition}

\label{proposition necessary condition for del}

If a \del \ subset $A$ of a recursively presented Polish space $(\ca{X},\ca{T})$ admits a good parameter in \del \ then the class \del \ is dense both in $A$ and in the complement $\ca{X} \setminus A$ with respect to \ca{T}.

\end{proposition}

\begin{proof}
Let \ep \ be a good parameter for $A$ in \del \ and $\ca{T}_{\infty}$ be the corresponding extension of the topology \ca{T}. Suppose that $V \in \ca{T}$ intersects $A$. Since $\ca{T}_{\infty}$ extends $\ca{T}$ we have that $V \in \ca{T}_{\infty}$. From the choice of $\ep$ and $\ca{T}_{\infty}$ we also have that $A \in \ca{T}_{\infty}$. Thus the set $V \cap A$ is a non-empty $\ca{T}_{\infty}$-open set. Since $(\ca{X},\ca{T}_{\infty})$ is recursively presented in \ep, the set $V \cap A$ contains a point say $x_0$ which is recursive in $\ep$. It follows that $\{x_0\}$ is a $\del(\ep)$ subset of $(\ca{X},d_{\infty})$ and therefore $\{x_0\}$ is a $\del(\ep)$ subset of $(\ca{X},d)$. Thus $x$ is $\del(\ep)$ point of $(\ca{X},d)$. Since $\ep \in \del$ we have that $x$ is a \del \ point of $(\ca{X},d)$.

The proof for $\ca{X} \setminus A$ is similar.
\end{proof}

\begin{remark}

\label{remark the parameter may not be in del}

\normalfont

It follows from the previous proposition that a \del \ set $A$ may not admit a good parameter in \del. To see this take $\ca{X} = \ca{N}$ and $A$ any non-empty $\Pi^0_1$ set with no \del \ members.
\end{remark}

Our next goal is to find sufficient conditions, under which a given \del \ set admits a good parameter in \del. We will work inside a special category of spaces.

\begin{definition}

\label{definition recursively zero-dimensional}

\normalfont

A recursively presented Polish space $(\ca{X},\ca{T})$ is \emph{recursively zero-dimensional} if there is a distance function $d$ on \ca{X}, which generates the topology \ca{T} witnesses that $(\ca{X},\ca{T})$ is recursively presented with recursive presentation the sequence $(r_i)_{\iin}$ and moreover the relation $I \subseteq \ca{X} \times \om \times \om$ defined by
    \[
    I(x,i,s) \eq d(x,r_i) < q_s
    \]
is recursive. By replacing ``recursive" with ``$\ep$-recursive" one defines the notion of an \emph{$\ep$-recursively zero-dimensional space}.

\end{definition}

It is well-known that every zero-dimensional Polish space is topologically isomorphic to a closed subset of the Baire space, c.f. 7.8 \cite{kechris_classical_dst}. The next lemma is the effective analogue of this statement.

\begin{lemma}

\label{lemma zero-dimensional is embedded}

For every recursively zero-dimensional Polish space $\ca{X}$ there is a recursive injection $f: \ca{X} \to \ca{N}$ such that the set $\ca{Y}:=f[\ca{X}]$ is in $\Pi^0_1(\ep)$ for some $\ep \in \Delta^0_2$. Moreover the inverse function $f^{-1}: \ca{Y} \to \ca{X}$ is computed by a semi-recursive subset of $\ca{N} \times \om \times \om$ on \ca{Y} i.e., there is a semi-recursive $R \subseteq \ca{N} \times \om \times \om$ such that for all $\alpha \in \ca{Y}$ we have that $d(f^{-1}(\alpha),r_i) < q_s \eq R(\alpha,i,s)$. In particular the inverse function $f^{-1}$ is continuous.
\end{lemma}

\begin{proof}
Fix a suitable distance function $d$ on \ca{X}. By replacing $d$ with $d(1+d)^{-1}$ we may assume that $d \leq 1$. First we claim that there is a recursive relation $A \subseteq \ca{X} \times \seq$ such that (1) $A_0 = \ca{X}$, (2) $A_{s} = \cup_{\n} A_{s \ \hat{} \ n}$, (3) $A_{s \ \hat{} \ n} \cap A_{s \ \hat{} \ m} = \emptyset$ for all $n \neq m$, (4) each $A_{s}$ has diameter less than $2^{-lh(s)}$, where $A_{s} = \set{x \in \ca{X}}{A(x,s)}$. Notice that each $A_{s}$ is a clopen set and that we do not exclude the case $A_{s} = \emptyset$.

We now prove this claim. For all $i,n \in \om$ we denote by $N(i,n)$ the open $d$-ball with center $r_i$ and radius $(n+1)^{-1}$, so that the relation $I \subseteq \ca{X} \times \om \times \om$ defined by $I(x,i,n) \eq x \in N(i,n),$ is recursive. Using Kleene's Recursion Theorem we obtain a recursive function $b: \ca{X} \times \seq \to 2$ such that $b(x,0) = 1$ for all $x \in \ca{X}$ and
\[
b(x,s \ \hat{} \ k) = 1 \eq b(x,s) = 1 \ \& \ x \in N(k, 2^{lh(s)+2}) \ \& \ b(r_k,s) = 1,
\]
for all $x \in \ca{X}$, $s \in \seq$ and $k \in \om$. So if we define $B_{s} = \set{x \in \ca{X}}{b(x,s) = 1}$ we have that $B_0 = \ca{X}$ and $B_{s \ \hat{} \ k} = B_{s} \cap N(k,2^{lh(s)+2})$ if $r_k \in B_s$ and $B_{s \ \hat{} \ k} = \emptyset$ otherwise, for all $s \in \seq$ and $k \in \om$. It is clear that the diameter of $B_{s}$ is less than $2^{lh(s)}$ and that it is a clopen set. It is also holds that $B_{s} = \cup_{\n} B_{s \ \hat{} \ n}$. To see this let $x \in B_s$; since $B_s$ is open we can choose $r_k \in B_s$ with $d(x,r_k) < 2^{-(lh(s)+2)}$. Then $x \in B_{s} \cap N(k,2^{lh(s)+2}) = B_{s \ \hat{} \ k}$. With one more application of the Recursion Theorem we get a recursive set $A \subseteq \ca{X} \times \seq$ such that $A_0 = \ca{X}$ and
\[
A_{s \ \hat{} \ k} = (B_{s \ \hat{} \ k} \setminus \cup_{i < k} B_{s \ \hat{} \ i}) \cap A_s
\]
for all $s \in \seq$ and $k \in \om$, where $A_s$ is the $s$th-section of $A$. It is clear that conditions (1)-(4) are satisfied for this family $(A_{s})_{s \in \seq}$.

Having proved our claim we define $f: \ca{X} \to \ca{N}$ as follows
\begin{eqnarray*}
f(x)(n) &=&  \emph{the unique $i$ for which there is (a unique) $s \in \seq$ such that}\\
        && \emph{$x \in A_{s \ \hat{} \ i}$ and $lh(s) = n$,}
\end{eqnarray*}
for all $x \in \ca{X}$ and \n. It is easy to check that $f(x)\upharpoonright n$ is the unique $s \in \seq$ of length $n$ such that $x \in A_s$ and that $f(x)(n)$ is the unique $i$ such that $x \in A_{f(x)\upharpoonright n \ \hat{} \ (i)}$. Clearly the function $f$ is recursive and injective.

We define now $\ep(s) = 1$ exactly when there exists $i$ such that $r_i \in A_s$ and $0$ otherwise. It is clear that $\ep \in \Delta^0_2$. Since each $A_{s}$ is clopen we have that $\ep(s) = 1$ exactly when $A_{s} \neq \emptyset$. Moreover one can verify that
\begin{eqnarray*}
\alpha \in f[\ca{X}] \eq (\forall n)[A_{\alpha \upharpoonright n} \neq \emptyset]
                     \eq (\forall n)[\ep(\alpha \upharpoonright n) = 1].
\end{eqnarray*}
So the set $\ca{Y}:=f[\ca{X}]$ is a $\Pi^0_1(\ep)$ subset of \ca{N}. Finally we prove the assertion about $f^{-1}$. Suppose that $\alpha \in \ca{Y}$ and $x = f^{-1}(\alpha)$. We claim that
\[
d(x,r_i) < q_s \eq (\exists n,j)[r_j \in A_{\alpha \upharpoonright n} \ \& \ d(r_j,r_i) < q_s - 2^{-n}].
\]
For the left-to-right-hand direction pick some $n,j \in \om$ such that $2^{-n+1} < q_s - d(x,r_i)$ and $r_j \in A_{\alpha \upharpoonright n}$. (The set $A_{\alpha \upharpoonright n}$ is non-empty since $\alpha \in \ca{Y}$.) Since both $r_j$ and $x = f^{-1}(\alpha)$ are members of $A_{\alpha \upharpoonright n}$ we have that $d(r_j,x) < 2^{-n}$. It holds then
$d(r_j,r_i) \leq d(r_j,x) + d(x,r_i) < \frac{1}{2^n} + d(x,r_i) < q_s - \frac{1}{2^n}.$
For the right-to-left-hand direction using that $x = f^{-1}(\alpha) \in A_{\alpha \upharpoonright n}$ for all $n$ we compute $d(x,r_i) \leq d(x,r_j) + d(r_j,r_i) < 2^{-n} + q_s - 2^{-n} = q_s$. Thus the equivalence is proved. Take now $R(x,i,s) \eq (\exists n,j)[r_j \in A_{\alpha \upharpoonright n} \ \& \ d(r_j,r_i) < q_s - \frac{1}{2^n}]$ and we are done.
\end{proof}

It would be interesting to see if the parameter $\ep$ in the previous proof can be chosen to be recursive.

\begin{remark}

\label{remark about the embedding}

\normalfont

Suppose that \ca{X} is a recursively-zero-dimensional Polish space and that \ca{Y} and $f: \ca{X} \to \ca{Y}$ are as in the previous lemma.

(1) Consider a recursive presentation for \ca{X}, say $(x_i)_{\iin}$ and define
$R(i,j) \eq x_i \neq x_j$. The metric space $(\ca{Y},p_\ca{N})$ admits a presentation which is recursive in $R$. To see this define $y_i = f(x_i)$ for all $i \in \om$ and notice that the sequence $(y_i)_{\iin}$ is dense in \ca{Y}. Moreover for all $i,j$ with $x_i \neq x_j$ there is exactly one triple $(u,n,m)$, where $u$ is a finite sequence of naturals and $n\neq m \in \om$, such that $x_i, x_j \in A_u$, $x_i \in A_{u \ \hat{} \ (n)}$ and $x_j \in A_{u \ \hat{} \ (m)}$. So the least natural $k$ for which $f(x_i)(k) \neq f(x_j)(k)$ is exactly the length of the finite sequence $u$. We define $g(i,j)=$ \textit{the length of the previous $u$}, if $R(i,j)$ and $g(i,j) = 0$, otherwise. The function $g$ is $R$-recursive and from the definition of $p_\ca{N}$ it is clear $p_\ca{N}(y_i,y_j) = p_\ca{N}(f(x_i),f(x_j)) = (g(i,j)+1)^{-1}$ for $x_i \neq x_j$.

(2) Consider the parameter $\ep \in \Delta^0_2$ of the proof of the previous lemma, so that \ca{Y} is in $\Pi^0_1(\ep)$. For every $A \subseteq \ca{X}$, if $A$ is in $\Pi^0_1$ the set $f[A]$ is a $\Pi^0_1(\ep)$ subset of \ca{N}, and if $A$ is in $\Sigma^0_{n+1}$ the set $f[A]$ is a $\Sigma^0_{n+1}$ subset of \ca{N} for all $n \geq 1$. To see this consider for all $k \in \om$ the space $\ca{X} \times \om^k$ (for $k=0$ we just mean the space \ca{X}) and the function $\tilde{f} : \ca{X} \times \om^k \to \ca{N} \times \om^k$ defined by $\tilde{f}(x,\vec{z}) = (f(x),\vec{z})$. Using that $\ca{Y}$ is a $\Pi^0_1(\ep)$ subset of \ca{N} one can show that for every $k \in \om$ and for every $A \subseteq \ca{X} \times \om^k$ in $\Pi^0_1$ the set $\tilde{f}[A]$ is a $\Pi^0_1(\ep)$ subset of $\ca{N} \times \om^k$. From this and the fact that $\ep \in \Delta^0_2$ one can show with induction on $n \geq 1$ that for all $k \in \om$ and all $A \subseteq \ca{X} \times \om^k$ in $\Sigma^0_{n+1}$ the set $\tilde{f}[A]$ is a $\Sigma^0_{n+1}$ subset of $\ca{N} \times \om^k$.

\end{remark}

The proof of Theorem \ref{theorem yiannis 4A7} (c.f. \cite{yiannis_dst}) shows that every \del \ subset of $\ca{N}$ is in fact the injective image under the projection of a $\Pi^0_1$ subset of $\ca{N} \times \ca{N}$. Thus in the case of \del \ subsets of \ca{N} the function $\pi$ of Theorem \ref{theorem yiannis 4A7} can be chosen to be an open mapping i.e., to carry open sets to open sets. This however does not necessarily imply that the inverse function $\pi^{-1}: A \to F$ is continuous. (We would need to know that $\pi$ is injective on the whole space \ca{N} in order to say this.) For reasons that will become clear latter on, the continuity of the inverse function will be necessary for our purposes. This property though poses a limitation on the \del \ set we start with.

\begin{remark}

\label{remark Gdelta is necessary so that the inverse function is continuous}

\normalfont

Suppose that $A \subseteq \ca{N}$ is in \del \ and that there is a $\Pi^0_1$ subset of \ca{N}, say $F$, and a recursive function $\pi: F \to \ca{N}$ such that $\pi[F] = A$, $\pi$ is injective on $F$ and the inverse function $\pi^{-1}: A \to F$ is continuous. Then $A$ is a $G_{\delta}$ subset of \ca{N}.

To see why the latter holds define the distance function $d$ on $A$ by $d(\alpha,\beta) = p_{\ca{N}}(\pi^{-1}(\alpha),\pi^{-1}(\beta))$ for all $\alpha, \beta \in A$. Since both functions $\pi$ and $\pi^{-1}$ are continuous, the topology on $A$ induced by $d$ is exactly \set{V \cap A}{V \ \textrm{open in} \ \ca{N}} i.e., the relative topology of \ca{N} on $A$. Moreover since $F$ is closed it follows that $(A,d)$ is Cauchy-complete. Therefore the set $A$ with the relative topology induced by \ca{N} is a Polish space. It follows from Theorem 3.11 in \cite{kechris_classical_dst} that $A$ is a $G_\delta$ set.
\end{remark}

What is perhaps more important is that the previous remark has an inverse.

\begin{lemma}

\label{lemma pi02 has continuous inverse}

For every $A \subseteq \ca{N}$ in $\Pi^0_2$ there is an $F \subseteq \ca{N} \times \ca{N}$ in $\Pi^0_1$ such that
\begin{eqnarray*}
A(\alpha) &\eq& (\exists \beta)F(\alpha,\beta)\\
          &\eq& (\exists \ \textrm{unique} \ \beta)F(\alpha,\beta)
\end{eqnarray*}
for all $\alpha \in \ca{N}$, and the function $pr^{-1}: A \to F: pr^{-1}(\alpha)= (\alpha,\beta)$, where $\beta$ is such that $F(\alpha,\beta)$, is continuous.
\vg{Similarly if $(\ca{X},\ca{T})$ is recursively zero-dimensional and $B$ is a $\Pi^0_2$ subset of \ca{X}, there is an $\ep \in \Delta^0_2$ and an $L \subseteq \ca{X} \times \ca{N}$ in $\Pi^0_1(\ep)$ satisfying all previous conclusions with $B$ and $L$ in the place of $A$ and $F$ respectively.}
\end{lemma}

As usual there is a similar version of Lemma \ref{lemma pi02 has continuous inverse} with respect to some parameter $\alpha$. For reasons of exposition we refrain ourselves from stating this parameterized version.

\begin{proof}
Since $A$ is in $\Pi^0_2$ using 3C.4 of \cite{yiannis_dst} there is a recursive $R \subseteq \ca{N} \times \om \times \om$ such that $A(\alpha) \eq (\forall n)(\exists m)R(\alpha,n,m)$ for all $\alpha \in \ca{N}$.

Define $F(\alpha,\beta) \eq (\forall n)[R(\alpha,n,\beta(n)) \ \& \ (\forall k < \beta(n)) \neg R(\alpha,n,k)]$ with $\alpha, \beta \in \ca{N}$. It is clear that $F$ is in $\Pi^0_1$ and that for all $\alpha \in \ca{N}$ there is at most one $\beta$ such that $F(\alpha,\beta)$. Moreover we have that $A(\alpha) \eq (\exists \beta)F(\alpha,\beta)$ for all $\alpha, \beta \in \ca{N}$. Now we prove that the inverse function $pr^{-1}: A \to F: pr^{-1}(\alpha)= (\alpha,\beta)$ is continuous. Clearly it enough to prove that the function $g: A \to \ca{N}: g(\alpha) = \beta$, where $\beta$ is such that $F(\alpha,\beta)$, is continuous. So let any $\alpha$ and $\beta$ with $F(\alpha,\beta)$ and $N \in \om$. For all \n \ we consider the set $V_n= \set{\alpha'}{R(\alpha',n,\beta(n)) \ \& \ (\forall k < \beta(n)) \neg R(\alpha',n,k)}$. Since $F(\alpha,\beta)$ and $R$ is recursive the set $V = \cap_{n < N} V_n$ is an open neighborhood of $\alpha$. It is then clear that for all $\alpha '\in V$ and all $\beta'$ with $F(\alpha', \beta')$ we have that $\beta '\upharpoonright N = \beta \upharpoonright N$.
\end{proof}

We are now ready for the next theorem.

\begin{theorem}

\label{theorem countable sets admit good parameter in del}

Suppose that $(\ca{X}, \ca{T})$ is a recursively zero-dimensional Polish space with no isolated points and that $A$ is a countable \del \ subset of \ca{X}. Then there exists an $\ep \in \del$ and a Polish topology $\ca{T}_{\infty}$ with suitable distance function $d_{\infty}$, which extends \ca{T} and has the following properties: \tu{(1)} the Polish space $(\ca{X},\ca{T}_{\infty})$ is
$\ep$-recursively presented, \tu{(2)} the set $A$ is a $\Delta^0_1(\ep)$ subset of $(\ca{X},d_{\infty})$ and \tu{(3)} every $B \subseteq \ca{X}$ is a $\del(\alpha)$ subset of $(\ca{X},d)$, where $d$ is a suitable distance function for $(\ca{X},\ca{T})$ and $\alpha \in \ca{N}$, exactly when $B$ is a $\del(\ep,\alpha)$ subset of $(\ca{X},d_{\infty})$.
\end{theorem}

\begin{proof}
For notational purposes we interchange $A$ with $\ca{X} \setminus A$, for it is the case of a co-countable \del \ set which is interesting. So let us assume that $\ca{X} \setminus A$ is countable. Since $A$ has a countable complement and it is \del \ there is a \del \ sequence $(x_n)_{\n}$ in \ca{X} such that $\ca{X} \setminus A = \set{x_n}{\n}$. Using this it is easy to find some $\ep_D \in \del$ such that $\ca{X}\setminus A$ is in $\Sigma^0_2(\ep_D)$ and so $A$ is in $\Pi^0_2(\ep_D)$. To see this we remark that $x \not \in A \eq (\exists n)(\forall s)[x \in N(\ca{X},s) \longrightarrow x_n \in N(\ca{X},s)]$. Define then $\ep_D(\langle n,s \rangle) = 1 \eq x_n \in N(\ca{X},s)$ and $0$ otherwise.

We first show how one can reduce the problem to subspaces of \ca{N}. We consider the embedding $f: \ca{X} \to \ca{N}$ of Lemma \ref{lemma zero-dimensional is embedded} and an $\ep_0 \in \Delta^0_2$ such that the set $\ca{Y} :=f[\ca{X}]$ is in $\Pi^0_1(\ep_0)$.\vg{It is clear that \ca{Y} is non-meager in \ca{N} (why??) and since $\ca{Y} \setminus f[A] = f[\ca{X} \setminus A]$ is a countable subset of \ca{N} it follows that $f[A]$ is non-meager as well.}  We take $\ep^* = \langle \ep_0, \ep_D \rangle \in \del$ and using (the parameterized version of) Remark \ref{remark about the embedding} we have that the set $f[A]$ is a $\Pi^0_2(\ep^*)$ subset of \ca{N}.

We consider the set \ca{Y} with the usual distance function of the Baire space $p_{\ca{N}}$. From Remark \ref{remark about the embedding} the space $(\ca{Y},p_\ca{N})$ is recursively presented in some arithmetical parameter, which we may assume that we have included in the previous $\ep^*$.\vg{Since \ca{Y} is $\Pi^0_1(\ep^*)$ subset of \ca{N} and $f[A]$ is a $\Pi^0_2(\ep^*)$ subset of \ca{N} it is easy to see that $f[A]$ is a $\Pi^0_2(\ep^*)$ subset of \ca{Y}.} Moreover the space $(\ca{Y},p_\ca{N})$ has no isolated points for otherwise $(\ca{X},d)$ would have had isolated points. The set $\ca{Y} \setminus f[A] = f[\ca{X} \setminus A]$ is a countable subset of \ca{Y} and since \ca{Y} has no isolated points it follows that $f[A]$ is co-meager in \ca{Y}.

Thus we have an $\ep^* \in \del$ such that the space $(\ca{Y},p_\ca{N})$ is recursively presented in $\ep^*$, the set $f[A]$ is a $\Pi^0_2(\ep^*)$ subset of \ca{N}, and $f[A]$ is co-countable and co-meager in \ca{Y}. We will show that there is a Polish topology $\ca{S}_{\infty}$ on \ca{Y} with suitable distance function $p_\infty$ and an $\ep \in \del$ such that all conclusions are satisfied for this \ep \ and for \ca{Y}, $p_\infty$, $\ca{S}_{\infty}$ and $f[A]$ in the place of \ca{X}, $d_\infty$, $\ca{T}_\infty$ and $A$ respectively. From there we can turn back to \ca{X} using the function $f$. More specifically we define $d_\infty(x,y) = p_\infty(f(x),f(y))$ for all $x,y \in \ca{X}$, i.e., we turn $f$ into an isometry. Then all conclusions are satisfied.

We now go back to the proof of Theorem \ref{theorem main} and we show that the $\ep_1$ and $\ep_2$ defined there, are in fact in \del. Since $f[A]$ is  a $\Pi^0_2(\ep^*)$ subset of \ca{N} from the parameterized version of Lemma \ref{lemma pi02 has continuous inverse} there is an $F_1 \subseteq \ca{N} \times \ca{N}$ in $\Pi^0_1(\ep^*)$, and a function $\pi_1: \ca{N} \times \ca{N} \to \ca{N}$ such that: (a) $\pi_1$ is $\ep^*$-recursive, (b) $\pi_1$ is injective on $F_1$, (c) $\pi_1[F_1] = f[A]$ and (d) the function $\pi_1^{-1}: f[A] \to F_1$ is continuous. In order to match with the notation of the proof of Theorem \ref{theorem main} we view $F_1$ as a subset of \ca{N}. We now show that the $\ep_1$ of the proof of Theorem \ref{theorem main} is in \del. Define the relation $R_1 \subseteq \seq$ by $R_1(s) \eq N_s \cap F_1 \neq \emptyset$,
where $N_s= \set{\alpha \in \ca{N}}{\alpha(i) = (s)_i \ \forall i < lh(s)}$ for all $s \in \seq$. We need to show that $R_1$ is \del. Since $\ep^* \in \del$ and $F_1$ is in $\Pi^0_1(\ep^*)$ and it is clear that $R_1$ is in \sig.

We now claim that $R_1(s) \eq (\exists \alpha \in \del)[\alpha \in N_s \cap F_1]$. Using the Theorem on Restricted Quantification (4D.3 in \cite{yiannis_dst}) it follows from the latter equivalence that $R_1$ is in \pii. The right-to-left-hand implication is clear, so let us assume that $R_1(s)$. Pick some $\alpha_0 \in N_s \cap F_1$. In particular $\alpha_0 = \pi_1^{-1}(y_0)$ for some $y_0 \in f[A]$. Since the function $\pi_1^{-1}: f[A] \to F_1$ is continuous there is some basic open (and thus recursive) neighborhood $V \subseteq \ca{N}$ such that $y_0 \in V$ and for all $y \in V \cap f[A]$ we have that $\pi_1^{-1}(y) \in N_s \cap F_1$. Since $f[A]$ is co-meager in \ca{Y} and $V \cap \ca{Y}$ is non-empty open in \ca{Y} we have that the set $V \cap f[A]$ is non-meager in \ca{Y}. Moreover $V \cap f[A]$ is a $\Pi^0_2(\ep^*)$ -and thus $\del(\ep^*)$- subset of \ca{Y}. From the parameterized version of 4F.20 of \cite{yiannis_dst} there is some $y \in V \cap f[A]$, which is a $\del(\ep^*)$ point of \ca{Y}. Since $\ep^* \in \del$ it is easy to check that $y$ is a \del \ point of \ca{N} and so $\alpha = \pi_1^{-1}(y) \in N_s \cap F_1$ is also in \del. Thus the equivalence is proved.

Now we deal with the complement of $f[A]$ in \ca{Y} and subsequently with $\ep_2$. From Theorem \ref{theorem yiannis 4A7} there is a set $F_2 \subseteq \ca{N}$ in $\Pi^0_1$ and a recursive function $\pi_2: \ca{N} \to \ca{X}$ such that $\pi_2[F_2] = \ca{Y} \setminus f[A]$ and $\pi_2$ is injective on $F_2$. Since $\ca{Y} \setminus f[A]$ is countable, so is $F_2$. It follows from the Effective Perfect Set Theorem 4F.1 in \cite{yiannis_dst} that $F_2$ consists of \del \ points. Hence $N_s \cap F_2 \neq \emptyset \eq (\exists \alpha \in \del)[\alpha \in N_s \cap F_2]$. From the latter equivalence it follows that $\ep_2 \in \del$. Now we take $\ep = \langle \ep^*,\ep_1,\ep_2\rangle \in \del$ and we continue as in the proof of Theorem \ref{theorem main}.
\end{proof}
Notice that in the previous proof there is no need to choose the function $\pi_2$ so that its inverse function is continuous. The reason for this is our hypothesis about countability which ensures that the corresponding closed set $F_2$ consists of \del \ points. Suppose for the moment that we drop the hypothesis about countability in favor of $A$ and its complement being $G_\delta$ and \del \ sets. Assume moreover that \del \ is dense in both $A$ and $\ca{X} \setminus A$. Then from (the parameterized version of) Lemma \ref{lemma pi02 has continuous inverse} we can always choose the previous functions $\pi_1$ and $\pi_2$ so that their inverses are continuous functions and then we can repeat the previous proof. The only problem with this idea is that we have to make sure that our application of Lemma \ref{lemma pi02 has continuous inverse} stays within the level of hyperarithmetical parameters. Under the hypothesis about countability that is clear, for a countable \del \ set is easily in $\Sigma^0_2(\ep^*)$ for some hyperarithmetical $\ep^*$. But now it is not so clear that the $G_\delta$ sets $A$ and $\ca{X} \setminus A$ that we start with (which are also in \del) are in fact in $\Pi^0_2(\ep^*)$ for some hyperarithmetical $\ep^*$. The latter assertion is true thanks to a deep result of Louveau, which states that if a subset $A$ of a recursively presented Polish space \ca{X} is \del \ and $\tboldsymbol{\Pi}^0_\xi$, for some $\xi < \om_1$, then there exists an $\ep \in \del$ such that $A$ is in $\Pi^0_\xi(\ep)$, c.f. \cite{louveau_a_separation_theorem_for_sigma_sets}.

\vg{\begin{theorem}{\normalfont (Louveau)}
\label{theorem louveau}
Suppose that \ca{X} is a recursively presented Polish space and that $A \subseteq \ca{X}$ is \del \ and $\tboldsymbol{\Pi}^0_\xi$, for some $\xi < \om_1$. Then there exists an $\ep \in \del$ such that $A$ is in $\Pi^0_\xi(\ep)$.
\end{theorem}
}

\vg{In the light of Louveau's Theorem we can see that we can start in Lemma \ref{lemma pi02 has continuous inverse} with a $G_\delta$ set $A$ which is \del. The only difference is that resulting set $F$ is in $\Pi^0_1(\ep)$ for some $\ep \in \del$.}

We can now proceed to the next theorem.

\begin{theorem}

\label{theorem G_delta and complement G_delta}

Suppose that $(\ca{X},\ca{T})$ is a recursively zero-dimensional Polish space and that $A$ is a \del \ subset of $\ca{X}$, which is also in $\tboldsymbol{\Delta}^0_2$. The following hold: \tu{(1)} the set $A$ admits a good parameter in \del \ if and only \del \ is dense in both $A$ and $\ca{X} \setminus A$ and \tu{(2)} if $A$ does not admit a good parameter in \del, then \param \ is hyperarithmetic in every good parameter for $A$.
\end{theorem}

\begin{proof}
From Louveau's Theorem there is an $\ep_D \in \del$ such that both sets $A$ and $\ca{X} \setminus A$ are in $\Pi^0_2(\ep_D)$. We consider the embedding $f: \ca{X} \to \ca{N}$ of Lemma \ref{lemma zero-dimensional is embedded} and an $\ep_0 \in \Delta^0_2$ such that the set $\ca{Y} :=f[\ca{X}]$ is in $\Pi^0_1(\ep_0)$. Consider also an arithmetical $\ep'_0$ such that the space $(\ca{Y},p_\ca{N})$ admits an $\ep'_0$-recursive presentation and define $\ep^* = \langle \ep_D, \ep_0, \ep'_0 \rangle$. It is clear that $\ep^* \in \del$, the space $(\ca{Y},p_\ca{N})$ is $\ep^*$-recursively presented and the sets $f[A]$ and $\ca{Y} \setminus f[A]$ are in $\Pi^0_2(\ep^*)$.

As pointed out in the proof of Theorem \ref{theorem countable sets admit good parameter in del} every good parameter for $f[A]$ (as a subset of \ca{Y}) is also a good parameter for $A$, for we can define the new distance function on \ca{X} in such a way that the function $f$ becomes an isometry. Using the same method one can see that the converse is also true modulo the parameter $\ep^*$ i.e., if $\ep$ is a good parameter for $A$ then $\langle \ep^*, \ep \rangle$ is a good parameter for $f[A]$. Thus $A$ admits a good parameter in \del \ exactly when $f[A]$ admits a good parameter in \del.

We apply the parameterized version of Lemma \ref{lemma pi02 has continuous inverse}. This yields sets $F_i \subseteq \ca{N}$ in $\Pi^0_1(\ep^*)$ and $\ep^*$-recursive functions $\pi_i : \ca{N} \to \ca{X}$, $i=1,2$, such that (a) $\pi_i$ is injective on $F_i$ for $i=1,2$, (b) $\pi_1[F_1] = f[A]$, $\pi_2[F_2] = \ca{Y} \setminus f[A]$, (c) the inverses $\pi_1^{-1}: f[A] \to F_1$ and $\pi_2^{-1}: \ca{Y} \setminus f[A] \to F_2$ are continuous.

We are now ready to prove (1). The left-to-right-hand implication is Proposition \ref{proposition necessary condition for del}, so we prove the inverse direction. Consider the previous functions $\pi_i$ and the sets $F_i$, $i=1,2$. We repeat the steps of the proof of Theorem \ref{theorem main}. Since \del \ is dense in $A$ and $f$ is recursive and injective, it is easy to check that \del \ is dense both in $f[A]$ and in $\ca{Y} \setminus A$. Using the argument about density as in the previous proof we can see that the parameters $\ep_1$ and $\ep_2$ (with the notation of the proof of Theorem \ref{theorem main}) are in fact in \del. Thus the fraction $\ep = \langle \ep^*, \ep_1, \ep_2 \rangle$ is  a good parameter for $f[A]$ which is in \del \ and so from the previous comments $A$ admits a good parameter in \del.

Now let us prove (2). Since $A$ (and hence $f[A]$) does not admit a good parameter in \del \ at least one of the relations $R_1, R_2 \subseteq \seq$ defined by $R_i(s) \eq N_s \cap F_i \neq \emptyset$ where $i=1,2$ and $N_s= \set{\alpha \in \ca{N}}{\alpha(j) = (s)_j \ \forall j < lh(s)}$, are not in \del. (For otherwise both $\ep_1$ and $\ep_2$ would be in \del \ and so as above $f[A]$ would admit a good parameter in \del.) Let us say that $R_1$ is not in \del. It is clear that $R_1$ is a \sig \ set. From a well-known theorem of Spector (c.f. \cite{spector_recursive_well-orderings})it follows that the characteristic function $\ep_1$ of $R_1$ has the same hyperdegree as the one of \param. Now let $\ep$ be any good parameter for $A$. We will prove that $R_1$ is in $\del(\ep)$. Notice that $\langle \ep^*, \ep \rangle$ is a good parameter for $f[A]$.

We claim that $R_1(s) \eq (\exists \alpha \in \del(\ep))[\alpha \in N_s \cap F_1]$. To see this assume that $N_s \cap F_1 \neq \emptyset$; therefore $\pi_1[N_s \cap F_1] \neq \emptyset$. Since the function $\pi_1^{-1}: f[A] \to F_1$ is continuous the set $\pi_1[N_s \cap F_1]$ is open in $f[A]$. Hence there is an open $V \subseteq \ca{N}$ such that $\pi_1[N_s \cap F_1] = V \cap f[A] \neq \emptyset$. The set $V \cap f[A]$ is open in the extended topology, which is $\langle \ep^*, \ep \rangle$-recursively presented, and therefore it contains a point, say $y_0$, which is recursive in $\langle \ep^*,\ep \rangle$. It follows that the singleton $\{y_0\}$ is a $\del(\langle \ep^*,\ep \rangle)$ subset of $(\ca{Y},p_{\infty})$, where $p_{\infty}$ is a suitable distance function for the extended topology. Therefore $\{y_0\}$ is a $\del(\langle \ep^*,\ep \rangle)$ subset of $(\ca{Y},p_\ca{N})$. Since $y_0 \in V \cap f[A] = \pi_1[N_s \cap F_1]$ there is some $\alpha_0 \in N_s \cap F_1$ such that $y_0 = \pi_1(\alpha_0)$. Using the fact that $\pi_1$ is injective on $F_1$ we have for $\beta \in \ca{N}$ that $\beta = \alpha_0 \eq \beta \in F_1 \ \& \ \pi_1(\beta) \in \{y_0\}$. Since $\pi_1$ is a recursive function it follows that the point $\alpha_0$ is in  $\del(\langle \ep^*, \ep \rangle) = \del(\ep)$ and the claim has been proved.

It follows from the Theorem of Restricted Quantification (4D.3 in \cite{yiannis_dst}) that $R_1$ is in $\pii(\ep)$. Since $R_1$ is in \sig \ it follows that $R_1$ is in $\del(\ep)$.
\end{proof}

\begin{corollary}
\label{corollary to G_delta and complement G_delta}
Suppose that $(\ca{X},\ca{T})$ is a recursively zero-dimensional Polish space and that $A$ is a \del \ subset of \ca{X}, which is also in $\tboldsymbol{\Delta}^0_2$. The set $A$ admits a good parameter which is either in \del \ or has the same hyperdegree as \param.
\end{corollary}

\begin{proof}
This is immediate from Theorems \ref{theorem main} and \ref{theorem G_delta and complement G_delta}.
\end{proof}

It would be interesting to see if the previous theorem can be extended to sets which are above the level of $\tboldsymbol{\Delta}^0_2$ sets. Before we close this section it is perhaps worth mentioning that the idea of using the $\Pi^0_1$ sets $F_i$ and the recursive functions $\pi_i$ is not always the best way to compute the complexity of the parameter, for in some cases it is easier to use directly the properties of the set that we start with.

\begin{proposition}

\label{proposition vector spaces}

Suppose that $(\ca{X}, \|\cdot\|)$ is a separable vector space (real or complex). We consider the induced distance function $d(x,y) =  \norm{x}{y}$, $x, y \in \ca{X}$ and we assume that the metric space $(\ca{X},d)$ is recursively presented. Then every open ball of \ca{X} admits a good parameter in \del.
\end{proposition}

\begin{proof}
It is enough to prove the conclusion for $A = \set{x \in \ca{X}}{d(x,0) < 1}$.
Define $d_A(x,y) = d(x,y) + |d(x,A)^{-1} - d(y,A)^{-1}|$ for all $x,y \in A$. Then $d_A$ is a distance function on $A$ which generates the same topology as the restriction of the topology of the norm on $A$. Moreover the metric space $(A,d_A)$ is complete. It is easy to find a \del \ sequence $(x_n)_{\n}$ in $A$ such that the set $\set{x_n}{\n}$ is dense in $A$. Then the sequence $(x_n)_{\n}$ is a presentation of $(A,d_A)$ which is recursive in some hyperarithmetical parameter.

Now we deal with the complement of $A$. Let $d_{A^c}$ be the restriction of the distance function $d$ on $\ca{X} \setminus A$. Since the latter set is closed it follows that $(\ca{X}\setminus A, d_{A^c})$ is complete and separable. Moreover it admits a presentation which is recursive in some hyperarithmetical parameter, for there is a \del \ sequence $(y_n)_{\n}$ with $\|y_n\| > 1$ for all \n, which is dense in $\ca{X} \setminus A$. Finally consider the direct sum $(A,d_A) \oplus (\ca{X} \setminus A, d_{A^c})$.
\end{proof}

\section{A Uniformity Result.}

We conclude this article with the uniformity result that we have mentioned in the Introduction.

\begin{definition}

\label{definition of codings of distance functions}

\normalfont

We denote by \distf \ the set of all distance functions on \om \ and by \completion{d} the completion of the metric space $(\om,d)$, where $d \in \distf$. It is clear that every complete and separable metric space, which is an infinite set, is isometric to a space of the form \completion{d} for some $d \in \distf$, so the set \distf \ can characterize up to isometry all complete and separable metric spaces. Every $d \in \distf$ can be encoded by a member of \ca{N}, for $d$ is just a double sequences of real numbers. In our case it is enough to work with distance functions on $\om$ which take only rational values; it is a bit shorter to encode these.

Suppose that $d \in \distf$ takes only rational values. We define $\beta_d \in \ca{N}$ as follows: $\beta_d(\langle i,j,m,n \rangle) = 1 \eq d(i,j) = m (n+1)^{-1}$ and $\beta_d(t)$ is $0$ in any other case. We say that $\beta \in \ca{N}$ \emph{encodes the complete and separable space $(\ca{X},d^\ca{X})$} if there is a distance function $d$ with rational values such that $\beta = \beta_d$ and $(\ca{X},d^\ca{X})$ is isometric to \completion{d}.

We will also use the following encoding of $y$-recursive functions from \ca{N} to \ca{N}, where $y$ varies through a recursively presented Polish space \ca{Y}, as given in 7A of \cite{yiannis_dst}. For all recursively presented Polish spaces \ca{X} consider the  $G^\ca{X} \subseteq \om \times \ca{X}$ as defined in 3H.1 of \cite{yiannis_dst}, which is in $\Sigma^0_1$ and universal for $\Sigma^0_1 \upharpoonright \ca{X}$. Fix now a recursively presented Polish space \ca{Y}. For every function $\pi: \ca{N} \to \ca{N}$ which is $y$-recursive for some $y \in \ca{Y}$ there is some $e \in \om$ such that for all $s \in \om$ we have that $\pi(\alpha) \in N(\ca{N},s) \eq G^{\ca{Y} \times \ca{N} \times \om}(e,y,\alpha,s)$. We consider such a number $e$ as a \emph{code of the function $\pi$}. We define the set $Cod^\ca{Y} \subseteq \om \times \ca{Y}$ by
$
Cod^\ca{Y}(e,y) \eq (\forall \alpha)(\exists \ \textrm{a unique} \ \beta)(\forall s)[\beta \in N(\ca{N},s) \longleftrightarrow G^{\ca{Y} \times \ca{N} \times \om}(e,y,\alpha,s)].
$
It is easy to verify that $Cod^\ca{Y}$ is a \pii \ set and so the set of codes of (total) $y$-recursive functions is in $\pii(y)$.
\end{definition}

\begin{theorem}

\label{theorem uniformity result}

Suppose that \ca{Z} is a Polish space, \ca{X} is a closed subset of \ca{N} and that $P$ is a Borel subset of $\ca{Z} \times \ca{X}$ such that every section $P_z$ is neither the empty set nor the whole space \ca{X}. Assume moreover that for all $z \in \ca{Z}$ one of the following cases applies to $P_z$: \tu{(A)} $P_z$ is countable, \tu{(B)} $P_z$ is co-countable, \tu{(C)} $P_z$ is a $\tboldsymbol{\Delta}^0_2$ set and $\del$ is dense both in $P_z$ and in $\ca{X} \setminus P_z$.

Then there is a Borel-measurable function $f: \ca{Z} \to \ca{N}$ such that for all $z \in \ca{Z}$ the following are true: \tu{(a)} there is a distance function $d^\ca{X}_z$ on \ca{X} such that the metric space $(\ca{X},d^\ca{X}_z)$ is complete and separable, \tu{(b)} $f(z)$ encodes the space $(\ca{X},d^\ca{X}_z)$, \tu{(c)} the topology induced by $d^\ca{X}_z$ extends the original one, \tu{(d)} the section $P_z$ is $d^\ca{X}_z$-clopen, and \tu{(e)} a subset of \ca{X} is Borel in the original topology exactly when it is Borel in the topology induced by $d^\ca{X}_z$.
\end{theorem}

\begin{proof}
Fix for the moment some $z \in \ca{Z}$ and let $F_i \subseteq \ca{N}$ be closed and $\pi_i: \ca{N} \to \ca{X}$ be continuous such that $\pi_i$ is injective on $F_i$ for $i=1,2$, $\pi_1[F_1] = P_z$ and $\pi_2[F_2] = \ca{X} \setminus P_z$. As usual we define the distance functions $d^1_z$ and $d^2_z$ on $P_z$ and $\ca{X} \setminus P_z$ respectively so that the $\pi_i$'s become isometries. The space $(\ca{X}, p_z):= (P_z,d^1_z) \oplus (\ca{X} \setminus P_z,d^2_z)$ satisfies the conclusions (c)-(e). Consider now sequences $(\alpha^i_n)_{\n}$ of distinct terms which are dense in $F_i$, $i=1,2$. It follows that the sequence $(\pi_1(\alpha^1_0), \pi_2(\alpha^2_0), \pi_1(\alpha^1_1), \pi_2(\alpha^2_1), \dots)$ is dense in $(\ca{X}, p_z)$. We define $d \equiv d(z) \in \distf$ by $d(2i,2j) = p_\ca{N}(\alpha^1_i,\alpha^1_j)$, $d(2i+1,2j+1) = p_\ca{N}(\alpha^2_i,\alpha^2_j)$ and $d(2i,2j+1) = d(2i+1,2j) = 2$. It is clear that \completion{d(z)} is isometric to $(\ca{X}, p_z)$. Notice that $d(z)$ takes only rational values. We consider the $\beta_{d(z)}$ which encodes $d(z)$ as in Definition \ref{definition of codings of distance functions}. We will show that there is a Borel way of choosing this $\beta_{d(z)}$.

Let us assume for simplicity that \ca{Z} is recursively presented, \ca{X} is recursively zero-dimensional and that $P$ is in \del, so that every section $P_z$ is in $\del(z)$. We will apply the Strong $\Delta$-Selection Principal  (4D.6 in \cite{yiannis_dst}), where the underlying pointclass is $\Gamma=\pii$. From our hypothesis and the construction in the previous proofs it follows that for all $z \in \ca{Z}$ there are $\ep \in \del(z)$, sets $F_i \subseteq \ca{N}$ in $\Pi^0_1(\ep,z)$ and $(\ep,z)$-recursive functions $\pi_i: \ca{N} \to \ca{X}$, $i=1,2$, such that $\pi_1[F_1] = P_z$, $\pi_2[F_2] = \ca{X} \setminus P_z$ and $\pi_i$ is injective on $F_i$, $i=1,2$. Moreover there are $\del(z)$ sequences $(\alpha^i_n)_{\n}$, which are dense in $F_i$ for $i=1,2$ and -as we may assume- consist of distinct terms.

We consider the space $Tr$ of trees on the natural numbers, c.f. (4.32) in \cite{kechris_classical_dst}. This is easily a recursively presented Polish space. Moreover every $F \subseteq \ca{N}$ in $\Pi^0_1(z)$ is the body $[T]$ of some $z$-recursive tree $T$; this is immediate from 4A.1 in \cite{yiannis_dst}. We define the relation $R_1 \subseteq \ca{Z} \times \om \times Tr \times \ca{N}^2$ as follows
\begin{eqnarray*}
R_1(z,e,T,\ep,\alpha) &\eq& \textrm{$T$ is an $(\ep,z)$-recursive tree }\\
                            &&\& \ e \ \textrm{encodes a total $(\ep,z)$-recursive function $\pi: \ca{N} \to \ca{N}$}\\
                            &&\ \ \ \textrm{which is injective on $[T]$}\\
                            &&\& \ \pi[[T]] = P_z \ \& \ ((\alpha)_i)_{\iin} \ \textrm{is dense in} \ [T]
\end{eqnarray*}

Similarly we define the relation $R_2$ by replacing $P_z$ with $\ca{X} \setminus P_z$. From our previous comments and from the fact that $P_z \neq \emptyset, \ca{X}$ it follows that for all $z \in \ca{Z}$ there are $(e_i,T_i,\ep_i,\alpha_i) \in \del(z)$ such that $R_i(z,e_i,T_i,\ep_i,\alpha_i)$ for $i=1,2$. Suppose for the moment that have proved that the $R_i$'s are \pii \ sets. Then from the Strong $\Delta$-Selection Principle there are \del-recursive functions $g=(g_1,g_2,g_3,g_4), h=(h_1,h_2,h_3,h_4): \ca{Z} \to \om \times Tr \times \ca{N}^2$ such that for all $z \in \ca{Z}$ we have that
$R_1(z,g(z))$ and $R_2(z,h(z))$. Let $(\alpha^1_i)_{\iin}$ and $(\alpha^2_i)_{\iin}$ be the sequences which arise from $g_4(z)$ and $h_4(z)$ respectively. As above we define $d \equiv d(z) \in \distf$ as follows: $d(2i,2j) = p_\ca{N}(\alpha^1_i,\alpha^1_j)$, $d(2i+1,2j+1) = p_\ca{N}(\alpha^2_i,\alpha^2_j)$ and $d(2i,2j+1) = d(2i+1,2j) = 2$. Finally we define $f(z) = \beta_{d(z)}$ for all $z \in \ca{Z}$. It is clear that $f$ is \del-recursive and thus it is Borel-measurable. According to the previous comments this function $f$ satisfies conclusions (a)-(e).

So it remains to verify that the sets $R_1$ and $R_2$ are indeed in \pii. We consider the set $Cod^{\ca{N} \times \ca{Z}} \equiv Cod$ as in the comments following Definition \ref{definition of codings of distance functions}. We also define the relation $Gr(e,\ep,z,\alpha,\beta) \eq (\forall s)[\beta \in N(\ca{N},s) \longleftrightarrow G^{\ca{N}^2 \times \ca{Z} \times \om}(e,\ep,z,\alpha,s)]$. In other words, when $Cod(e,\ep,z)$ holds, $Gr(e,\ep,z,\alpha,\beta)$ means that $\beta$ is the image of $\alpha$ under the $(\ep,z)$-recursive function encoded by $e$. Clearly $Gr$ is in \del. It is now easy to see that $e$ encodes an $(\ep,z)$-recursive function $\pi: \ca{N} \to \ca{N}$ which is injective on $[T]$ and $\pi[[T]] = P_z$ exactly when:

(i) $Cod(e,\ep,z) \ \& \ (\forall \alpha,\beta)[\alpha \in [T] \ \& \ Gr(e,\ep,z,\alpha,\beta) \longrightarrow \beta \in P_z]$ and

(ii) $\forall (\alpha_1,\alpha_2,\beta_1,\beta_2)(\forall i=1,2)[Gr(e,\ep,z,\alpha_i,\beta_i) \ \& \ \alpha_1 \neq \alpha_2 \longrightarrow \beta_1 \neq \beta_2]$ and

(iii) $(\forall \beta)(\exists \alpha \in \del(\ep,\beta))[\beta \in P_z \longrightarrow Gr(e,\ep,z,\alpha,\beta)]$.

(Here we are using that for all $\beta \in P_z$ the unique $\alpha \in [T]$ for which $\pi(\alpha)= \beta$ is in $\del(\ep,z,\beta)$.) Using a similar method with universal sets one can prove that the relation $Q_1 \subseteq \ca{N} \times \ca{Z} \times Tr$ defined by $Q_1(\ep,z,T) \eq T$ \emph{is $(\ep,z)$-recursive}, is in \del. Moreover it is easy to check using the definition that the relation $Q_2 \subseteq \ca{N} \times Tr$ defined by $Q_2(\alpha,T) \eq ((\alpha)_s)_{s \in \om}$ \emph{is dense in} $[T]$, is in \pii. It follows that $R_1$ is in \pii. Similarly one shows that $R_2$ is in \pii.
\end{proof}

\bibliographystyle{amsplain}

\bibliography{my_bibliography}

\end{document}